\documentclass
[11pt]{article}

\usepackage[a4paper, margin=1in]{geometry}
\usepackage{amsfonts,amsmath,amssymb,amsthm,graphicx,fixmath,latexsym,color}
\usepackage{xcolor}
\usepackage{hyperref,epsfig}
\usepackage{algorithmic}
\usepackage{algorithm}
\usepackage{xspace}

\providecommand{\remove}[1]{}

\newcommand{\F}{\mathcal{F}}
\newcommand{\B}{\mathcal{B}}
\newcommand{\M}{\mathcal{M}}
\newcommand{\h}{\mathcal{H}}

\newtheorem{theorem}{Theorem}[section]

\newtheorem{proposition}[theorem]{Proposition}

\newtheorem{question}[theorem]{Question}
\newtheorem{definition}[theorem]{Definition}

\newtheorem{observation}[theorem]{Observation}
\newtheorem*{theorem*}{Theorem}
\newtheorem*{lemma*}{Lemma}
\newtheorem*{proposition*}{Proposition}

\begin{document}

\title{Blockers for Simple Hamiltonian Paths in Convex Geometric Graphs of
Even Order}

\author{Chaya Keller\thanks{Department of Mathematics, Ben-Gurion University of the NEGEV, Be'er-Sheva, Israel. \texttt{kellerc@math.bgu.ac.il}. Research partially supported by Grant 1136/12 from the Israel Science Foundation
and by the Hoffman Leadership and Responsibility Program at the Hebrew University.} \mbox{ }
and Micha A. Perles\thanks{Einstein Institute of Mathematics, Hebrew University, Jerusalem, Israel.
\texttt{perles@math.huji.ac.il}}
}

\maketitle

\begin{abstract}

Let $G$ be a complete convex geometric graph on $2m$ vertices, and let $\mathcal{F}$ be a family of subgraphs of $G$. A \emph{blocker} for $\mathcal{F}$ is a set of edges, of smallest possible size, that meets every element of $\mathcal{F}$. In~\cite{KP12} we gave an explicit description of all blockers for the family of simple perfect matchings (SPMs) of $G$. In this paper we show that the family of simple Hamiltonian paths (SHPs) in $G$ has exactly the same blockers as the family of SPMs. Our argument is rather short, and provides a much simpler proof of the result of~\cite{KP12}.

\end{abstract}

\section{Introduction}

In this paper we consider convex geometric graphs, i.e., graphs whose vertices are points in convex position in the plane, and whose edges are segments connecting pairs of vertices. Let $G=CK(2m)$ be the complete convex geometric graph of order $2m$. All graphs we consider throughout the paper are subgraphs of $G$, and thus, we slightly abuse notation identifying a graph with its set of edges.
\begin{definition}
A \emph{simple perfect matching} (SPM) in $G$ is a set of $m$ pairwise disjoint edges (i.e., edges that do not meet, not even in an interior point). A \emph{simple Hamiltonian path} (SHP) in $G$ is a non-crossing path of length $2m-1$ (i.e., containing all $2m$ vertices and $2m-1$ edges).
\end{definition}

For a family $\F$ of subgraphs of $G$, a natural Tur\'{a}n-type question is: what is the maximal possible number of edges in a geometric graph on $2m$ vertices that does not contain any element of $\F$? This question was extensively studied with respect to various families $\F$, e.g., all sets of $k$ disjoint edges~\cite{K79,KP96} and all sets of $k$ pairwise crossing edges~(\cite{CP92}, and see also~\cite{BKV03}).

\medskip

\noindent An equivalent way to state the question is to consider sets that ``block'' all elements of $\F$:
\begin{definition}
A set of edges in a geometric graph $G$ is called a blocking set for $\F$
if it intersects (i.e., contains an edge of) every element of $\F$. A \emph{blocker} for $\F$ is a blocking set of smallest possible size. The family of blockers for $\F$ is denoted $\B(\F)$.
\end{definition}
Using this formulation, the above question is equivalent to the question:
\begin{question}\label{Q1}
What is the size of the blockers for $\F$?
\end{question}
In various cases, including the two cases mentioned above, the answer to Question~\ref{Q1} is known, and then, the natural desire is to provide a \emph{characterization} of the blockers for $\F$.

\medskip

\noindent In~\cite{KP12} we provided a complete characterization of the blockers for the family $\M$ of SPMs of $G$. The characterization involves the notion of a \emph{caterpillar tree}~\cite{Caterpillar1}.
\begin{definition}
A tree $T$ is a caterpillar if the derived graph $T'$ (i.e., the graph obtained from $T$ by removing all leaves and
their incident edges) is a path (or is empty). A geometric caterpillar is simple if it does not contain a pair of crossing edges. A longest (simple) path in a caterpillar $T$ is called a spine of $T$.
\end{definition}
\begin{theorem}[\cite{KP12}]
Let $V$ be the set of vertices of $G$ (viewed as the vertex set of a convex polygon $P$ in the plane), labelled cyclically from $0$ to $2m-1$. Any blocker for $\M$ is a simple caterpillar tree whose spine lies on the boundary of the polygon and is of length $t \geq 2$. If the spine ``starts'' with the vertex $0$ and
the edge $[0,1]$, then the edges of the blocker are:
\begin{equation}\label{Eq:Blocker-old}
\{[i-1,i]:1 \leq i \leq t\} \cup
\{[t+j-1-\epsilon_{t+j},t+j+\epsilon_{t+j}]:1 \leq j \leq m-t\},
\end{equation}
where the $\epsilon_i$ are natural numbers satisfying $1 \leq
\epsilon_{t+1} < \epsilon_{t+2} < \ldots < \epsilon_m \leq m-2$.

Conversely, any set of $m$ edges of the described form is a blocker for $\M$. \label{Thm:SPMs}
\end{theorem}

The geometric interpretation of~\eqref{Eq:Blocker-old} is as follows. If the polygon $P$ is regular
(which can be assumed w.l.o.g.), then the direction of each consecutive
edge of the blocker, as listed above, is obtained from the
direction of the preceding edge by rotation by $\pi/m$ radians. In
the first $t$ edges, the ``back'' endpoint of each edge is the
``front'' endpoint of the previous edge. Starting with the
$(t+1)$-st edge, the ``back'' endpoint goes ``back'' (as reflected
by subtraction of the corresponding $\epsilon_i$), and the length
of the edge changes accordingly. An example of a blocker for $\M$ is presented in Figure~\ref{fig:T1}.

\medskip

The proof of Theorem~\ref{Thm:SPMs} uses only elementary tools, but is rather complex, spanning about
15 pages.

\bigskip

In this paper we provide a complete characterization of the blockers for the family $\h$ of all simple Hamiltonioan paths (SHPs). As in any SHP, the odd-labelled edges form an SPM, it is clear that any blocking set for $\M$ is also a blocking set for $\h$. It is easy to show that both the blockers for $\M$ and the blockers for $\h$ are of size $m$ (see Section~\ref{sec:Preliminaries}). Hence, we clearly have $\B(\M) \subset \B(\h)$. We show that these two families of blockers are actually equal.
\begin{theorem}\label{Thm:Main}
Let $G$ be the complete convex geometric graph on $2m$ vertices. Denote by $\B(\M)$ and $\B(\h)$ the families of blockers for the families $\M$ of simple perfect matchings and $\h$ of simple Hamiltonian paths in $G$, respectively. Then $\B(\h)=\B(\M)$.
\end{theorem}
The proof of Theorem~\ref{Thm:Main} is rather short, and in particular, it yields a short proof for Theorem~\ref{Thm:SPMs} -- much shorter than the proof presented in~\cite{KP12}.

The characterization of blockers for SHPs in a complete convex geometric graph of odd order $2m+1$ is more complicated, and will be presented in a subsequent paper. 

\begin{figure}[tb]
\begin{center}
\scalebox{0.6}{
\includegraphics{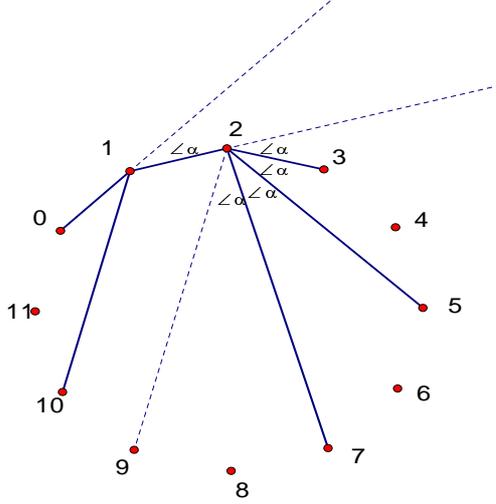}
} \caption{A blocker in a convex geometric graph on $12$ vertices with spine of length $t=3$. The
edges of the blocker are drawn as full (not dotted) lines. In
the notation of Theorem~\ref{Thm:SPMs}, $\epsilon_4=1$,
$\epsilon_5=2$, and $\epsilon_6=4$. The angle $\alpha$ is $\pi/6$
radians. The diagonal $[2,9]$ is parallel to the diagonal
$[1,10]$, and helps to depict the angle between the diagonals
$[2,7]$ and $[1,10]$.} \label{fig:T1}
\end{center}
\end{figure}

\section{Notations and Observations}
\label{sec:Preliminaries}

We view the vertices of $G$ as the vertices of a convex polygon $P$ of order $2m$. We label them, in order, by the numbers $0,1,2,\ldots,2m-1$. We regard the labels as elements of the cyclic group $\mathbb{Z}_{2m}=\mathbb{Z}/2m\mathbb{Z}$, and define the \emph{direction} of the edge $[i,j]$ to be $i+j (\bmod 2m)$. For $0 \leq k < 2m$, we denote by $D(k)$ the set of all edges in direction $k$. Two edges of the same direction are called \emph{parallel}. The \emph{order} of an
edge $[i,i+k]$ (where the addition is modulo $2m$) is $\min(k,2m-k)$.

Note that for odd $k$, the set $D(k)$ is an SPM. For even $k$, $D(k)$ is an ``almost perfect simple matching'', containing $m-1$ edges. Clearly, $E(G)$ is the disjoint union of $D(0),D(1),\ldots,D(2m-1)$. If the vertices of $G$ form a regular $2m$-gon, then each $D(k)$ is indeed a set of parallel line segments.

\medskip

The sets $D(2i)$ ($0 \leq i<m$) are $m$ pairwise disjoint SPMs. Similarly, the sets $D(2i) \cup D(2i+1)$ $(0 \leq i<m)$ are $m$ pairwise disjoint SHPs. It follows that the blockers for SPMs/SHPs in $G$ must be of size $\geq m$. As shown in the `easy direction' of the result of~\cite{KP12}, the blockers for SPMs are indeed of size $m$. Since each SHP includes an SPM as mentioned above, this implies that the blockers for SHPs are also of size $m$.

Let $B$ be a blocker for $\h$. The $m$ edges of $B$ belong to at most $m$ directions (out of the $2m$ possible directions). Since the union of any two consecutive directions (i.e., $D(2i-1) \cup D(2i)$ or $D(2i) \cup D(2i+1)$) is an SHP, $B$ cannot miss two consecutive directions. This leaves only two choices: either exactly one edge in each odd direction $D(2i-1)$, or exactly one edge in each even direction $D(2i)$. The second choice leaves all boundary edges $[2i-1,2i],[2i,2i+1]$ untouched, and these form a simple Hamiltonian circuit. Thus, we are left with the first choice only: one edge in each odd direction. We have thus proved:
\begin{observation}\label{Obs1}
Any blocker for $\h$ consists of $m$ edges: one edge in each odd direction.
\end{observation}

\section{Proof of Theorem~\ref{Thm:Main}}
\label{sec:proof}

Let $B$ be a blocker for $\h$. By Observation~\ref{Obs1} we know that $B$ consists of one edge in each odd direction. Our goal is to show that $B$ satisfies the conditions of Theorem~\ref{Thm:SPMs} (and thus, belongs to $\B(\M)$). Since $B$ blocks every Hamiltonian path on the boundary of the convex hull, we know that $B$ contains at least two boundary edges. We also know that $B$ cannot contain two opposite boundary edges, as they belong to the same direction. Our first step will be to show that the boundary edges of $B$ form a single (consecutive) path.
\begin{proposition}\label{Prop1}
Let $B$ be a set of $m$ edges and let $2 \leq k \leq m-1$ (hence, $1 \leq m-k \leq m-2$). Suppose that $B$ contains the boundary edges $g=[0,1]$ and $f=[m+k-1,m+k]$ but does not contain any of the $m-k$ boundary edges $[m+k+j,m+k+j+1]$ ($0 \leq j < m-k$) in between. Then $B \not \in \B(\h)$.
\end{proposition}

\begin{proof}
Define $h=[2m-1,0]$. We will construct $k-1$ SHPs $P_1,P_2,\ldots,P_{k-1}$, such that:
\begin{enumerate}
\item $f \not \in P_i$, $g \not \in P_i$, but $h \in P_i$, for each $i$, $1 \leq i \leq k-1$.

\item No two $P_i$'s share an edge of odd order, other than $h$.

\item All the edges of odd order in $P_i$ are in direction $2i+1$ or $2i-1$.

\item All the edges of even order in $P_i$ are in direction $2i$.
\end{enumerate}
This will prove that $B$ is not a blocker for all SHPs. Indeed, otherwise $B$ meets each $P_i$ in at least one edge. As $B$ contains only edges of odd order and $h \not \in B$, it follows from (2) that $B$ meets $P_1,P_2,\ldots,P_{k-1}$ in $k-1$ different edges. By (3), all these edges are in directions $1,3,5,\ldots,2k-1$. But by (1), the two edges $g$ (in direction $1$) and $f$ (in direction $2k-1$) of $B$ do not belong to any $P_i$. It follows that $k-2$ edges of $B$ meet the $k-1$ $P_i$'s, contrary to (2).

\medskip

\noindent The $k-1$ special SHPs are defined as follows. $P_i$ contains five batches of edges, denoted $A,B,C,D,E$: First, $m-1$ edges in direction $2i$, in two batches:
\[
A=\{[i-1,i+1],[i-2,i+2],\ldots,[0,2i]\} \mbox{ ($i$ edges), and }
\]
\[ 
D = \{[2m-1,2i+1],[2m-2,2i+2],\ldots,[m+i+1,m+i-1]\} \mbox{ ($m-1-i$ edges)}.
\]
Second, $i$ edges in direction $2i+1$:
\[
B = \{[i,i+1],[i-1,i+2],\ldots,[1,2i]\}.
\] 
Third, $m-1-i$ edges in direction $2i-1$:
\[
E = \{[2m-2,2i+1],[2m-3,2i+2],\ldots,[m+i,m+i-1]\}.
\]
Finally, a single edge $h$ in direction $2m-1$: $C=\{[0,2m-1]\}$. 

\medskip \noindent Batches $B$ and $A$ together form a path from $i$ to $0$. Batches $D$ and $E$ together form a path from $2m-1$ to $m+i$. These two paths are connected by $h$ (batch $C$) to a single path from $i$ to $m+i$ (see Figure~\ref{fig:2}).


It is clear that the $P_i$'s satisfy conditions~(1)--(4) above, which completes the proof.
\end{proof}

\begin{figure}[tb]
\begin{center}
\scalebox{1.1}{
\includegraphics{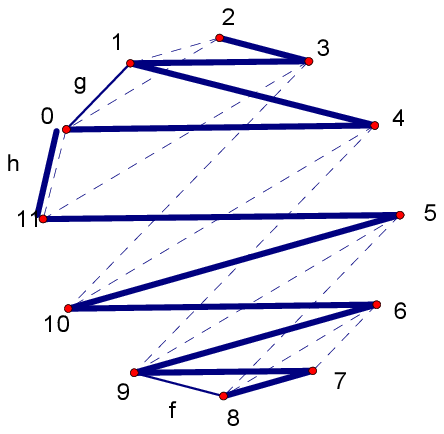}
} \caption{An illustration for the proof of Proposition~\ref{Prop1}. In the figure, $m=6$ and $k=3$, hence $g=[0,1]$,
$f=[8,9]$, and $h=[11,0]$. The SHP $P_1$ is depicted in punctured lines, and the SHP $P_2$ is depicted in bold lines.} \label{fig:2}
\end{center}
\end{figure}

Now we are ready to prove our main theorem.

\begin{proof}[Proof of Theorem~\ref{Thm:Main}]
Let $B \in \B(\h)$. By Proposition~\ref{Prop1}, we may assume that the boundary edges of the convex hull of $P$ in $B$ form a path $\langle 0,1,\ldots,j \rangle$ of length $j$, $2 \leq j \leq m$. As $B$ contains exactly one edge in each odd direction (by Observation~\ref{Obs1}), we know that $B$ contains $m-j$ further edges in directions $2j+1,2j+3,\ldots,2m-1$. All these edges are proper diagonals of $P$. We show several properties of these edges, which will imply that $B$ satisfies the conditions of Theorem~\ref{Thm:SPMs}.
\begin{enumerate}
\item None of the extra edges connects two vertices of the boundary path $\langle 0,1,\ldots,j \rangle$. Indeed, any edge of odd order that connects two vertices of that path is parallel to an edge of that path, and we know that $B$ cannot contain two edges of the same direction.

\item Suppose one of the extra edges connects two vertices that are not on the boundary path $\langle 0,1,\ldots,j \rangle$, or connects an endpoint of that boundary path with a vertex outside the path. To be specific, assume this edge is $[s,t]$, where $j \leq s<t<2m$ and $s \not \equiv t (\bmod 2)$. In this case, we construct a specific SHP $P_0$ that avoids $B$, yielding a contradiction. To construct $P_0$, we concatenate the boundary path $\langle s,s+1,\ldots,t \rangle$ (of length $t-s$) with the zig-zag path $\langle t,s-1,t+1,s-2,t+2,s-3,\ldots \rangle$ (of length $2m-1-(t-s)$), as demonstrated in Figure~\ref{fig:3}. The boundary path $\langle s,s+1,\ldots,t \rangle$ clearly avoids $B$. The zig-zag part uses only edges that are either parallel to $[s,t]$ (which belongs to $B$) but not the edge $[s,t]$ itself, or edges of even direction that are certainly not in $B$. Hence, $B$ avoids $P_0$.

\begin{figure}[tb]
\begin{center}
\scalebox{1.1}{
\includegraphics{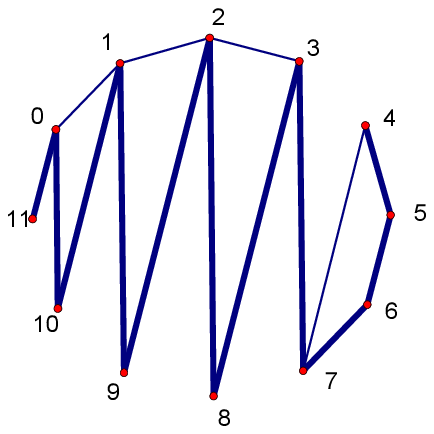}
} \caption{The SHP $P_0$ used in the proof of Theorem~\ref{Thm:Main}. In the figure, $j=3$, $s=4$, and $t=7$.
The edges of $P_0$ are drawn in bold.} \label{fig:3}
\end{center}
\end{figure}

\item Finally, assume two of the extra $m-j$ edges of $B$ are $[\alpha,\beta]$, $[\alpha',\beta']$, where
    \[
    0<\alpha<\alpha'<j, \qquad j<\beta < 2m, \qquad j<\beta' < 2m, \qquad \beta-\beta' \leq \alpha'-\alpha
    \]
    (and possibly even $\beta' \leq \beta$). If $\beta-\beta'=\alpha'-\alpha$, then these two edges are parallel (i.e., $\alpha+\beta=\alpha'+\beta'$), a contradiction. Thus we may assume that $\beta-\beta'<\alpha'-\alpha$, i.e., $\alpha+\beta<\alpha'+\beta'$, and therefore, $\alpha'+\beta' \geq \alpha+\beta+2$, since both $\alpha+\beta$ and $\alpha'+\beta'$ are odd numbers.

    We claim that $\alpha'+\beta'<2m$. Indeed, $\alpha' \leq j-1$ and $\beta' \leq 2m-1$, hence $\alpha'+\beta' \leq 2m+j-2$. If $\alpha'+\beta'>2m$, then $[\alpha',\beta']$ is parallel to one of the edges on the boundary path $\langle 0,1,\ldots,j \rangle$ which is impossible. Therefore, $\alpha'+\beta'<2m$.

    Now we construct an SHP $P_1$ that avoids $B$, yielding a contradiction. To do so, we construct three paths, as shown in Figure~\ref{fig:4}:

\begin{figure}[tb]
\begin{center}
\scalebox{1.1}{
\includegraphics{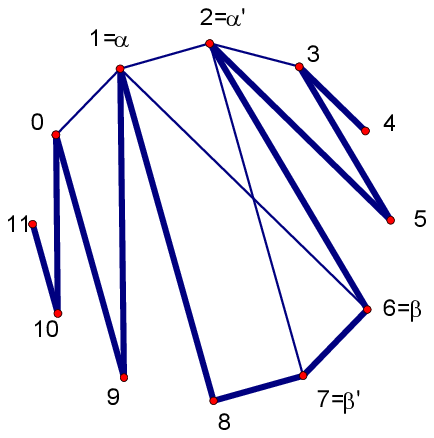}
} \caption{The SHP $P_1$ used in the proof of Theorem~\ref{Thm:Main}. In the figure, $j=3$, $\alpha=1$, $\alpha'=2$, $\beta=6$, and $\beta'=7$. The edges of $P_1$ are drawn bold.} \label{fig:4}
\end{center}
\end{figure}

    \begin{itemize}
    \item The zig-zag path $\langle \beta, \alpha+1, \beta-1,\alpha+2,\beta-2,\ldots, \rangle$, consisting (alternately) of edges of even direction $\alpha+\beta+1$ and odd direction $\alpha+\beta$, covering all vertices of the arc $\langle \alpha+1, \alpha+2,\ldots,\beta-1,\beta \rangle$. This path is traversed \emph{backwards}. Note that this path avoids $B$, since all its edges of odd direction are parallel, but not equal, to $[\alpha,\beta]$.

    \item The boundary path $\langle \beta,\beta+1,\ldots,\beta'+\alpha'-\alpha \rangle$. Note that $(\beta'+\alpha'-\alpha)-\beta \geq 2$. In addition, $\beta'+\alpha'-\alpha \leq 2m-2$, since $\beta'+\alpha' <2m$ and $\alpha \geq 1$. Hence, the path is well-defined. This path clearly avoids $B$.

    \item The zig-zag path $\langle \beta'+\alpha'-\alpha,\alpha, \beta'+\alpha'-\alpha+1, \alpha-1,\ldots \rangle$, consisting (alternately) of edges of odd direction $\beta'+\alpha'$ and edges of even direction $\beta'+\alpha'+1$, covering all vertices of the arc $\langle \beta'+\alpha'-\alpha, \beta'+\alpha'-\alpha+1,\ldots,2m-1,0,1,\ldots,\alpha \rangle$. Note that this path avoids $B$, since its edges of odd direction are all parallel (but not equal) to $[\alpha',\beta']$.
\end{itemize}
\end{enumerate}
We have thus shown that the $m-j$ extra (i.e., non-boundary) edges of $B$, of directions $2j+1,\ldots,2m-1$, must connect interior vertices of the boundary path $\langle 0,1,\ldots,j \rangle$ with interior vertices of the complementary path $\langle j,j+1,\ldots,m-1,0 \rangle$. Moreover, the ``roots'' of these edges (i.e., their intersections with $\langle 0,1,\ldots,j \rangle$) form a weakly monotone decreasing function of the direction. This is just another way of stating the conditions of Theorem~\ref{Thm:SPMs}. This completes the proof.
\end{proof}

\end{document}